\documentclass{article}
\usepackage[margin=1.35in]{geometry}

\usepackage{color,graphicx,  amsthm,  natbib}

 \usepackage{amsmath,amssymb,amsfonts}
\usepackage{bbm}

\usepackage{mathtools}
\usepackage{nicematrix}

\usepackage[colorlinks=true,allcolors=blue]{hyperref}

\newtheorem{thm}{Theorem}%
\newtheorem{lem}[thm]{Lemma}
\newtheorem{clm}[thm]{Claim}

\theoremstyle{definition}

\title{A note on the distinct distances problem over finite fields}
\author{
Nataly Brukhim\thanks{Institute for Advanced Study, Princeton NJ. Email: natalybrukhim@gmail.com} \and
Ariel Bruner\thanks{Princeton University, Princeton NJ.
		Email: arielbruner@gmail.com} \and
Orit E. Raz\thanks{Institute for Advanced Study, Princeton NJ, and Ben-Gurion University of the Negev, Be\'er-Sheva, Israel,\\ \indent Email:  oritraz@bgu.ac.il}}
\date{March 2025}

\begin{document}

\maketitle
\begin{abstract}
    We study a finite-field analogue of the Erd\H{o}s distinct distances problem under the Hamming metric. For a set \(S\subseteq \mathbb{F}_q^n\)   let $\Delta(S)$
denote the set of Hamming distances determined by \(S\). We prove the lower bound
\[
|\Delta(S)| \;\ge\; \frac{\log |S|}{2\log(2nq)},
\]
and show this bound is tight when \(|S|=O(\text{poly}(n))\), where the constant of proportionality depends only on $q$. We then also study the problem of finding a large \emph{rainbow set}, that is, a subset \(S\subseteq \mathbb{F}_q^n\) for which all \(\binom{|S|}{2}\) pairwise Hamming distances spanned by $S$ are distinct. 
In contrast to the Euclidean setting,  we show that a set with many distinct distances does not imply the existence of a large rainbow set,  by giving an explicit construction.
 Nevertheless, we establish the existence of large rainbow sets, and prove that every large set in \(\mathbb{F}_q^n\) necessarily contains a non-trivial rainbow subset.

\end{abstract}

\vspace{1ex}

\section{Introduction}
The classical distinct distances problem, posed by \cite{erdos1946}, asks how many distinct distances must be determined by a set of $n$ points in the Euclidean plane. Erd\H{o}s conjectured that such a set always spans at least $\Omega(n/\sqrt{\log n})$ distinct distances. In a breakthrough result, \cite{GuthKatz} nearly resolved the problem, closing the gap up to logarithmic factors. Since then, numerous variants of this problem have been investigated, including higher-dimensional analogues, distances measured with respect to norms other than the Euclidean norm, and configurations where the points are constrained to lie on a constant-degree algebraic curve (\cite{Alon25}, \cite{GuthKatz}, \cite{pach2014distinct}, \cite{charalambides2014distinct}, \cite{mathialagan2023distinct},  \cite{raz2020note}, \cite{aronov2003distinct} and \cite{sheffer2014distinct}).
 
Analogous questions have also been studied over finite fields. Let $\mathbb{F}_q$ denote the finite field of order $q \ge 2$. In \cite{iosevich2007erdos}, the authors consider subsets $S \subset \mathbb{F}_q^n$ and investigate the number of distinct squared-Euclidean distances determined by $S$. They show that if $S$ is sufficiently large, namely, if $|S|\ge Cq^{(n+1)/2}$, then $S$ spans at least $cq$ distinct squared-Euclidean distances, for some absolute constants $C,c>0$.

In this paper we consider an analogue of Erd\H{o}s's distinct distances problem over finite fields, but under the {\it Hamming distance}, which is a central notion in coding theory. 
For two vectors $x = (x_1,...,x_n)$ and $y = (y_1,...,y_n)$ in $\mathbb{F}_q^n$, the {\it Hamming distance} is defined by 
$$d_H(x,y) = \sum_{i=1}^n \mathbbm{1}[x_i \neq y_i].$$  Note that $0\le d_H(x,y)\le n$. For a set $S\subseteq \mathbb{F}_q^n$, we denote by 
\[
\Delta(S)=\{\, d_H(x,y): x,y\in S \,\},
\]
the set of Hamming distances spanned by $S$. 
A natural question is how small $\Delta(S)$ can be. In this direction, 
\cite{yazici2019hamming} proved that if $4|n$ and if $S\subset \mathbb{F}_q^n$ satisfies $$|S|>\frac{q^{n-1}}{n}\binom{n}{n/2}\binom{n/2}{n/4},$$ then $\Delta(S)$ contains all even integers  in $[n]$. In particular, $|\Delta(S)|\ge n/2$.
A stronger quantitative bound was later obtained by \cite{xu2020erdos}:
\begin{thm}[\cite{xu2020erdos}]
    Let $q\ge 4$ be a prime power, let $n\in \mathbb{N}$, and let $0<\alpha<1$. Then there exists $0<\beta=\beta(\alpha)< 1$ such that the  following holds. If $S\subset \mathbb{F}_q^n$ satisfies
    $
|S|\ge q^{\beta n},$
then $|\Delta(S)|\ge \alpha n$.
\end{thm}
This reflects the general philosophy, familiar from the Euclidean case: 
The number of distinct distances spanned by a set $S$ typically grows with $|S|$.
It is worth noting that all of the results above concern sets of size exponential in $n$. For smaller sets, however, the behavior can be very different. A prominent example is the {\it simplex code}, which has size of order $n$ but yet spans only a single distance (see the next section for details).

We establish the following general result. %
\begin{thm}\label{thm:large_set_many_distances}
Let $q\in \mathbb{N}$ with  $q\geq 2$ and let $n \in \mathbb{N}$. Then, for every $S\subseteq [q]^n$ one has
$$
|\Delta(S)|\ge \frac{\log |S|}{2\log(2nq)}.
$$
\end{thm}

Observe that our bound is asymptotically tight for sets of size linear in $n$, as can be seen by 
applying the theorem to the simplex code mentioned above. 
By generalizing the simplex code construction, we establish that Theorem~\ref{thm:large_set_many_distances} is in fact tight for sets of size polynomial in $n$, and therefore cannot be improved in general within this regime.

\begin{thm}\label{thm:exist_few_distances}
    Let $q\ge 2$ be a prime power and let $k, n\in \mathbb{N}$ with $n\geq  k$.
    Then, there exists a set $S\subset \mathbb{F}_q^n$ of size $\Theta(n^k)$ such that $|\Delta(S)|=O(1)$.
\end{thm}

Next, we consider the problem of finding large {\it rainbow sets}:
a set $R\subseteq \mathbb{F}_q^n$ is called {\it rainbow} if all $\binom{|R|}{2}$ Hamming distances spanned by pairs of points from $R$ are distinct.
For a set $S\subset \mathbb{F}_q^n$, let $\rho(S)$ denote the size $|R|$ of its largest rainbow subset $R\subset S$.

Distinct distances and rainbow sets are closely related: if a set 
$S$ contains a large rainbow subset, then 
$S$
necessarily spans many distinct distances. In particular,
$$
|\Delta(S)|\ge \binom{\rho(S)}{2}.
$$
In the Euclidean setting this connection becomes even stronger.
The lower bound on $|\Delta(S)|$ is typically obtained by bounding the number of quadruples $p,p',q,q'\in S$ satisfying $\|p-q\|=\|p'-q'\|$. As shown in \cite{charalambides2013note}, 
this allows one to pick a random subset of $S$ avoiding all quadruples $p,p',q,q'$ as above, therefore forming a rainbow subset of $S$. 
Consequently, any improvement in the lower bound for $|\Delta(S)|$ (which comes from a refinement of the upper bound on the number of such quadruples),
directly translates into a stronger lower bound for $\rho(S)$.

In contrast, over finite fields the situation is different, as the following result shows.
\begin{thm}\label{thm:gap}
    Let $n \in \mathbb{N}$. Then, there exists a set $S \subseteq \mathbb{F}_2^n$ such that $|\Delta(S)| = \Theta(\sqrt{n})$ and $\rho(S) = 2$. 
\end{thm}

Nevertheless, we prove that rainbow sets of substantial size do in fact exist in $\mathbb{F}_2^n$.
\begin{thm}\label{thm:existence_huge_rainbow}
Let $\epsilon > 0$ and integer $n > n_0(\epsilon)$.
    There exists a rainbow set $R\subset \mathbb{F}_2^n$ with $|R| \ge \big(\frac{1}{\sqrt{2}} - \epsilon\big)\sqrt{n}$.
\end{thm}

Moreover, rainbow subsets cannot be entirely avoided: every sufficiently large set must contain one of nontrivial size.
\begin{thm}\label{thm:large_set_large_rainbow}
    Let  $0\le \alpha\le 1$ and let $S\subseteq \mathbb{F}_2^n$ such that $|S|\geq \alpha 2^n$.  Then $$
    \rho(S)=\Omega\left(\alpha^{1/2} n^{1/2}\right).$$
\end{thm}

In our proofs we use the Frankl--Wilson theorem, and other classical results from extremal combinatorics, as well as explicit constructions that rely on coding theory, specifically the simplex code family. \\

Our paper is organized as follows. In Section~\ref{sec:simplex} we review the classical construction of simplex codes and prove Theorem~\ref{thm:exist_few_distances}, in Section~\ref{sec:manydistances} we prove Theorem~\ref{thm:large_set_many_distances}, and finally Theorems \ref{thm:gap}, \ref{thm:existence_huge_rainbow}, and \ref{thm:large_set_large_rainbow} are proven in Section~\ref{sec:rainbow}.

\section{Simplex code and Theorem \ref{thm:exist_few_distances}}\label{sec:simplex}
We introduce a tool from coding theory that will aid in proving our results. Coding theory seeks to characterize the maximum size of a code subject to prescribed length  and minimum distance. For further background on coding theory we refer the reader to  \cite{lint1999introduction, huffman2010fundamentals}. 

A {\it linear code} is a linear subspace  $\mathcal{C}\subset \mathbb{F}_q^n$. The {\it weight} $w(x)$ of a word $x\in \mathcal{C}$
is the number of nonzero entries of $x$. 
Note that by linearity of $\mathcal{C}$, the set of distances spanned by $\mathcal{C}$ is coincides with the set of weights of its nonzero elements.

The 
{\it $q$-ary simplex code} of dimension 
$m$,
denoted $S_q(m)$, is defined as the dual of the 
$q$-ary {\it Hamming code}.
Concretely, 
let $H$ be the $m\times \frac{q^m-1}{q-1}$ matrix, whose columns are representatives of the $\frac{q^m-1}{q-1}$ distinct one-dimensional subspaces of $\mathbb{F}_q^m$. (For $q=2$, this is simply all nonzero vectors in $\mathbb{F}_2^m$.)
Then 
$S_q(m)$ is precisely the row space of $H$, together with the zero vector.

The code has the following properties:
\begin{itemize}
    \item $S_q(m)$ is an $m$-dimensional subspace of $\mathbb{F}_q^n$, where $n=\frac{q^m-1}{q-1}$
    \item Every nonzero codeword in $S_q(m)$ has weight exactly $q^{m-1}$.
\end{itemize}
In particular,
\begin{equation}
\text{$|S_q(m)|=q^m=\Theta(n)$ and $|\Delta(S_q(m))|=1$.}
\end{equation}
For further details, see~\cite[Theorems 1.8.3, 2.7.5]{huffman2010fundamentals}.

We now prove Theorem~\ref{thm:exist_few_distances}.
\begin{proof}[Proof of Theorem \ref{thm:exist_few_distances}]
  Let $k\in \mathbb{N}$ be a constant parameter.
  We construct a set $S\subset\mathcal{F}_q^n$, with $n\geq k$, such that 
  $|S|=\Theta(n^k)$ and
  $|\Delta(S)|=O(1)$.

  Indeed,
let $m\in\mathbb{N}$ be the largest such that $n_0:=k(q^m-1)/(q-1)$ is an integer satisfying $n_0 \le n$.  Note that $n_0=\Theta(n)$. We now construct a $n$-dimensional set $S$ by concatenating the simplex code along with padded zeroes, i.e., 
  $$
  S:=S_q(m)\times\cdots\times S_q(m)\times 0^{n-n_0}\subset\mathbb{F}_q^{n},$$
  where $S_q(m)$ stands for the $q$-ary $m$-dimensional simplex code, $0^{n-n_0}$ is the zero vector of length $n-n_0$, and the product over codes is taken $k$ times.

By the properties of $S_q(m)$, for all $a_i \neq a_j \in S_q(m)$ we have $d_H(a_i,a_j) = q^{m-1}$. Let $x,y \in S$ and denote $x = (a_1,...,a_k)$ and $y = (b_1,...,b_k)$ where each $a_i$ and $b_i$ are codewords of $S_q(m)$. Then, we have, 
    $$
    d_H(x,y) = \sum_{j=1}^{n} \mathbbm{1}[x_j \neq y_j] = \sum_{i=1}^{k} \mathbbm{1}[a_i \neq b_i] \cdot q^{r-1}.
    $$
    Therefore, we have that 
    $$|\Delta(S)| = |\Delta(\mathbb{F}_2^{k})| =k.$$
Noting also that $$
|S|=|S_q(m)|^k
=q^{mk}=\Theta({n_0}/k)^k=\Theta(n/k)^k,$$
  this completes the proof of the theorem.
\end{proof}

\section{Proof of Theorem~\ref{thm:large_set_many_distances}}\label{sec:manydistances}

We start by proving a special case $q=2$ of Theorem \ref{thm:large_set_many_distances}, stated below.
\begin{lem}\label{thm:binary:large_set_many_distances}
Let $n > 2$ and  
$
S\subseteq \mathbb{F}_2^n$. Then 
$$
|\Delta(S)|\ge \frac{\log |S|}{2\log(n)}.
$$
\end{lem}

Lemma~\ref{thm:binary:large_set_many_distances} is a consequence of the following theorem of Ray-Chaudhuri and Wilson~\cite{ray1975t}: 
\begin{thm}[\cite{ray1975t}]\label{thm:frankl_wilson}
Let 
$\mathcal{F}$ be a family of $r$-element subsets of $[n]$. Let $\Delta(\mathcal{F})=\{|F\cap F'|\mid F,F'\in\mathcal{F}\}$ denote all distinct sizes of intersection between elements in $\mathcal{F}$, and put $d:=|\Delta(\mathcal{F})|$.
Then $$
|\mathcal{F}|\le \binom{n}{d}
.$$
\end{thm}

\begin{proof}[Proof of Lemma~\ref{thm:binary:large_set_many_distances}]
Let $S\subset\mathbb{F}_2^n$ and set $d:=|\Delta(S)|$.
Fix any $x\in S$ and consider the distances between $x$ and the other elements of $S$. Clearly, there are at most $d$ of them. Thus $S$ is contained in a union of at most $d$ spheres centered at $x$. By the pigeonhole principle, there exists a distance $r$, and a subset $S'\subset S$ such that 
\begin{equation}\label{eq:S'size}|S'|\ge (|S|-1)/d,
\end{equation} and $S'$ lies on the $r$-sphere centered at $x$. 
Note that we may assume, without loss of generality, that $x=0$. 
Indeed, shifting the set $S$ by $-x$ does not change the number of distinct distances spanned; that is, $|\Delta(S)|=|\Delta(S-x)|$. 

We identify, in the standard manner, each  $y\in S'$ with an $r$-subset $F_y\subset [n]$, whose elements are the non-zero coordinates of $y$. 
Observe that for $y_1,y_2\in S'$ we have
$$d_H(y_1,y_2)=2(n-|F_{y_1}\cap F_{y_2}|).$$
Thus 
$$
d':=
|\Delta(S')|=|\Delta(\mathcal{F})|,$$
where $\mathcal{F}=\{F_y\mid y\in S'\}$.
Applying Theorem~\ref{thm:frankl_wilson} to $\mathcal{F}$ we conclude that
$$
 |S'|\le \binom{n}{d'}\le n^{d'}.$$
Using $1\le d'\le d\le n$ and the inequality \eqref{eq:S'size}, 
 we get 
 $$
 \frac{|S|}{2d}\le n^d,
 $$ 
and by taking logarithm and re-arranging terms we get,
$$
\frac{\log(|S|) - 1 - \log(d)}{\log(n)}\le d,
$$
and the left-hand side can be then lower bounded by $\frac{\log(|S|)}{2\log(n)}$  for all $d \le \sqrt{|S|}/2$, which we can assume as otherwise the claim trivially holds.
 \end{proof}

We are now ready prove Theorem \ref{thm:large_set_many_distances}. 
\begin{proof}[Proof of Theorem~\ref{thm:large_set_many_distances}]
Assume first that $q=2^m$, for some $m\ge 2$. 
We will construct an embedding of $\mathbb{F}_q^n$ into $\mathbb{F}_2^{N}$, for $N:=n(q-1)$, that scales all distances by the same factor. 

Let $S_2(m)\subset \mathbb{F}_2^{q-1}$ be the binary $m$-dimensional simplex code (see Section~\ref{sec:simplex}).
Note that $|S_2(m)|=2^m=q$, and write its codewords as 
$S_2(m) = \{a_1,...,a_q\}$. We may then associate  with each $x\in \mathbb{F}_q$ a unique element $a_x\in S_2(m)$.
Define $\tau: \mathbb{F}_q^n \rightarrow \left(S_2(m)\right)^n$
by 
$$
(x_1,\ldots,x_q)\mapsto (a_{x_1},\ldots,a_{x_q}).
$$
It is easy to see that $\tau$ is an injection. Observe also the for every $x,y\in \mathbb{F}_q^n$ we have 
\begin{align*}
    d_H(\tau(x),\tau(y)) &= \sum_{j=1}^{N} \mathbbm{1}[\tau(x)_j \neq \tau(y)_j]\\
    &=
\frac{q}{2}\sum_{i=1}^{n} \mathbbm{1}[a_{x_i} \neq a_{y_i}]  \\
&=
\frac{q}{2}\sum_{i=1}^n \mathbbm{1}[x_i \neq y_i] \\
&=\frac{q}{2} d_H(x,y).
\end{align*}

In particular, for every $S\subset \mathbb{F}_q^n$, we have  
$$
|S|=|\tau(S)|$$
and 
$$|\Delta(S)|=|\Delta(\tau(S))|\ge \frac{\log |S|}{2\log (n(q-1))},$$
where the inequality is due to Lemma~\ref{thm:binary:large_set_many_distances}. 
This proves the theorem for $q=2^m$.

Finally, if $q$ is not a power of 2, we take a minimal $q'$, such that $q'=2^m$ for some power $m$, and $q\le q'$. Note that such $q'$ exists and satisfies 
$$
q\le q'\le 2q-1.
$$
As a set we can embed $[q]=\{0,1,...,q-1\}\subset \mathbb{F}_{q'}$, and apply the previous argument. 
This  gives
 $$|\Delta(S)|   \ge \frac{\log |S|}{2\log(n(q'{-}1))}  \ge \frac{\log |S|}{2\log(n(2q{-}1))},$$
which completes the proof of the theorem. 
\end{proof}

\section{Rainbow sets}
\label{sec:rainbow}
In this section,   we address the problem of finding large rainbow sets; subsets \(S\subseteq \mathbb{F}_q^n\) for which all \(\binom{|S|}{2}\) pairwise Hamming distances are distinct. 
Specifically, 
we prove Theorem~\ref{thm:gap}, showing that sets with many distances do not necessarily contain rainbow sets. Then, we prove Theorem~\ref{thm:existence_huge_rainbow}, establishing the existence of very large rainbow sets, and Theorem~\ref{thm:large_set_large_rainbow} showing that any sufficiently large set must also contain a non-trivial rainbow set.

\subsection{Proof of Theorem \ref{thm:gap}}

Denote $[n]:=\{0,1,\ldots,n-1\}$. First, for every $n\ge 2$, we  describe a construction of a set 
$A\subset [n+1]^n$, such that $|\Delta(A)|=n$ and $|\rho(A)|=2$. 
For $i=1,\ldots, n+1$, define $a_i\in [n+1]^n$ to be given by 
\begin{equation}\notag
    \forall j\in [n].(a_i)_j = 
    \begin{cases} 
        i-j &  j \leq i \\ 
        0 & \text{otherwise.} 
    \end{cases}
    \label{eq:case_equation}\end{equation}

\begin{clm}\label{clm:An}
    Let $1\le i_1<i_2\le n+1$. Then 
    $$
    {\rm dist}_H(a_{i_1},a_{i_2})=i_2.
    $$
\end{clm}
\begin{proof}
We count of the number of indices $1\le j\le n$ for which $(a_{i_1})_j\neq (a_{i_2})_j$.

Let $1\le j\le i_1<i_2\le n+1$.
Then
$$
(a_{i_2})_j=i_2-j\neq i_1-j=(a_{i_1})_j.$$
Similarly, for $1\le i_1< j< i_2\le n+1$, we have
$$(a_{i_2})_j=i_2-j\neq 0=(a_{i_1})_j.$$ 
Finally, if $i_2<n+1$, then for $1\le i_1<i_2\le j\le n$, we have
$$
(a_{i_1})_j=(a_{i_2})_j=0.$$ 
This proves the claim.\end{proof}

Letting
$$
A:=\{a_1,\ldots,a_{n+1}\},
$$
the claim show that 
$$
|\Delta(A)|=n.
$$

We next show that $\rho(A)=2$.
Suppose that $R\subset A_n$ is rainbow, and let $$
m:=\max\{1\le i\le n+1\mid a_i\in R\}.$$
By Claim~\ref{clm:An}, we have that for every $a\in R\setminus \{a_m\}$ we have 
${\rm dist}_H(a_m,a)=m$.
Since $R$ is a rainbow, this implies that necessarily $|A_n\setminus\{a_m\}|\le 1$. Thus $|R|\le 2$. 

To finish the proof, we use the argument from the proof of Theorem~\ref{thm:large_set_many_distances} to embed $[n+1]^n$ into $(\mathbb{F}_2)^{N}$, where $n^2\le N\le n(2n+1)$.
Thus $A$ is embedded to a subset, $\tau(A)\subset (\mathbb{F}_2)^N$, and we have 
$$|\Delta(\tau(A))|=n=\Omega( N^{1/2})$$
and $$\rho(\tau(A))=2.$$
This completes the proof of the theorem.
\hfill$\square$

\subsection{Proof of Theorem \ref{thm:existence_huge_rainbow}}
Our proof relies on the following result:
\begin{thm}[\cite{erdos1941problem}]\label{thm:erdos1941problem}
Let $\epsilon > 0$ and integer $n > n_0(\epsilon)$.  Then, there exists $B \subseteq  \{1,...,n\}$ of size at least $\big(\frac{1}{\sqrt{2}} - \epsilon\big)\sqrt{n}$ elements such that all of their pairwise sums are distinct. 
\end{thm}
 
Recall that the weight $w(x)$ of an element $x\in \mathbb{F}_2^n$ is defined as the number of nonzero entries of $x$.
Let 
$x_i\in \mathbb{F}_2^n$ be  the vector with prefix of $i$ ones followed by $n-i$ consecutive zeros.
Note that  \begin{equation}\label{weightxi}w(x_i)=i\end{equation} and that for $0\le i,j\le n$ we have
\begin{equation}\label{distxixj}
d_H(x_i, x_j) = |w(x_i) - w(x_j)|
= |i-j|.
\end{equation}
Let $B\subset [n]$ be the set given by Theorem~\ref{thm:erdos1941problem}, and define
$$
R:=\{x_i\mid i\in B\}.
$$
We claim that $R$ is a rainbow.
Indeed, let $x_i,x_j,x_k,x_\ell\in R$ such that $x_i\neq x_j$, $x_k\neq x_\ell$, and $\{x_i,x_j\}\neq \{x_k,x_\ell\}$. Equivalently,
\begin{equation}\label{distinctpairs}
    \text{$i\neq j$, $k\neq \ell$, and $\{i,j\}\neq \{k,\ell\}$.}
    \end{equation}
By properties of $B$ we have
\begin{equation}\label{distinctsums}
j+\ell\neq i+k,
\end{equation}
unless $\{j,\ell\}=\{i,k\}.$
Note that the latter occurs only if either $i=j$ and $k=\ell$, or $\{i,j\}=\{k,\ell\}$, therefore contradicting \eqref{distinctpairs}. Thus \eqref{distinctsums} holds, and in view of \eqref{weightxi} and \eqref{distxixj}, we get

$$
d_H(x_i,x_j)=w(x_j)-w(x_i)=j-1\neq \ell-k=w(x_\ell)-w(x_k)=d_H(x_k,x_\ell).$$
Thus $R$ is a rainbow, which completes the proof of the theorem. \hfill$\square$

\subsection{Proof of Theorem~\ref{thm:large_set_large_rainbow}}
Recall that the weight $w(x)$ of an element $x\in \mathbb{F}_2^n$ is defined as the number of nonzero entries of $x$.
Define a partial order on $\mathbb{F}_2^n$ based on index inclusion, such that for $x,y\in \mathbb{F}_2^n$, $X\leq y$ if and only if for every $i\in [n]$ we have $(x)_i\leq (y)_i$, where $(x)_i$ and $(y)_i$ stand for the $i$th coordinate of $x$ and $y$ respectively. Observe that if $x\le y$ then $d_H(x,y)=w(y)-w(x)$. 
Recall that a {\it chain} in a partial order is a set of pairwise comparable elements. 

\begin{lem}\label{lem:pigeonhole}
Let $S\subset \mathbb{F}_2^n$ be of size $|S|\ge \alpha 2^n$.
Then there exists $a\in \mathbb{F}_2^n$ such that the set $S+\{a\}$  contains a chain  $C$ of size $|C|\ge \alpha n$.
\end{lem}
\begin{proof}
Let $\mathcal{C}$ denote the set of all maximal chains in $\mathbb{F}_2^n$. Let $\Pi_n$ stand for the group of permutations of $[n]$. 
Note that the elements of $\Pi_n$ are in one to one correspondence with the set $\mathcal{C}$. Indeed, associate with $\pi=(i_1,\ldots,i_n)\in \Pi_n$ the chain $C_\pi=\{a_1<a_2<\cdots<a_n\}$, where for $k\in [n]$, $a_k\in\mathbb{F}_2^n$ is the vector with ones at entries $i_1,\ldots,i_k$ and zeros elsewhere. In particular, $|\mathcal{C}|=n!$.

Note that $x\in \mathbb{F}_2^n$, with weight $w(x)$, is a member of exactly $w(x)!(n-w(x))!$ maximal chains in $\mathcal{C}$. Indeed, these correspond to all the permutations $\pi$  in which the indices of the nonzero coordinates of $x$ appear before the rest of the indices. 
This implies that 
\begin{equation}\label{average}
    \sum_{\substack{a\in \mathbb{F}_2^n\\\pi\in \Pi_n}}
\left | (S+\{a\}) \cap C_\pi\right | =n|S|n!.
\end{equation}
Indeed,
\begin{align*}
\sum_{\substack{a\in \mathbb{F}_2^n\\\pi\in \Pi_n}}
\left | (S+\{a\}) \cap C_\pi\right | &=
\sum_{x\in S}
\sum_{ a\in \mathbb{F}_2^n}
\sum_{\pi\in\Pi_n}
\left | \{x+a\} \cap C_\pi\right | \\
&= 
\sum_{x\in S}\sum_{ a\in \mathbb{F}_2^n}
\left | \{\pi\in\Pi_n |x+a\in C_\pi\}\right |\\
&=
\sum_{x\in S}\sum_{ a\in \mathbb{F}_2^n}
w(x+a)!(n-w(x+a))! \\
&=
|S|\sum_{ a\in \mathbb{F}_2^n}
w(a)!(n-w(a))! \\
&= 
|S|
\sum_{ w\in[n]}
\binom{n}{w}w!(n-w)!\\
&= 
|S|
\sum_{ w\in[n]}n!\\
&=n|S|n!.
\end{align*}

Equation \eqref{average} then implies that, there exist $a\in \mathbb{F}_2^n$ and $\pi\in\Pi_n$ such that 
$$
|(S+\{a\})\cap C_\pi|
\ge
n|S|/2^n\ge \alpha n.
$$
Letting $C:=(S+\{a\})\cap C_\pi$, this proves the lemma. 
\end{proof}

Let $a$ and $C$ be given by Lemma~\ref{lem:pigeonhole}.
Observe that a subset $R\subset S$ is a rainbow in $S$ if and only if $R+\{a\}$ is a rainbow in $S+\{a\}$.
Since $C\subset S+\{a\}$, it suffices to prove that $C$ has a large rainbow subset. 
Our proof relies on the following result, which generalizes Theorem~\ref{thm:erdos1941problem}.
\begin{thm}[\cite{komlos1975linear}]\label{thm:komlos1975linear}
Let $A\subset \mathbb{N}$ be a finite set of size $n$. Then there exists a subset $B\subset A$ such that $|B|=\Omega(n^{1/2})$ and all the pairwise sums of elements of $B$ are distinct. 
\end{thm}

To apply the theorem we we embed the elements of $C$ in $\mathbb{N}$ by $x\in C\mapsto w(x)$. Since $C$ is a chain this is indeed an injection. Let $A:=\{w(x)\mid x\in C\}$.
By Theorem~\ref{thm:komlos1975linear}, there exists $B\subset A$, such that $|B|=\Omega(\alpha^{1/2}n^{1/2})$ and all the pairwise sums of elements of $B$ are distinct. 
Similar to the argument in the proof of Theorem~\ref{thm:existence_huge_rainbow}, this implies that the pre-image of $B$ is a rainbow subset of $C$. This completes the proof of Theorem~\ref{thm:large_set_large_rainbow}.\hfill$\square$

\bibliographystyle{apalike}
\bibliography{bib}

\end{document}